\documentclass[12pt]{amsart}
\usepackage{mathdots}
\usepackage[margin=1.25in]{geometry}

\usepackage{amsmath}
\usepackage{amsfonts}
\usepackage{amssymb}
\usepackage{amscd}
\usepackage{graphicx}
\usepackage[abbrev,alphabetic]{amsrefs}
\RequirePackage[dvipsnames,usenames]{color}
\usepackage{soul,xcolor}
\setstcolor{red}
\usepackage{stmaryrd}
\usepackage{multicol}

\newcommand\hidden[1]{} 
\usepackage{mathtools}
\newtagform{hidden}[\hidden]{}{}

\usepackage{hyperref}

\usepackage{amsthm}
\usepackage{comment}
\usepackage[all,cmtip]{xy}
\usepackage{tikz-cd}
\usetikzlibrary{cd}

\makeatletter
\def\@tocline#1#2#3#4#5#6#7{\relax
  \ifnum #1>\c@tocdepth 
  \else
    \par \addpenalty\@secpenalty\addvspace{#2}%
    \begingroup \hyphenpenalty\@M
    \@ifempty{#4}{%
      \@tempdima\csname r@tocindent\number#1\endcsname\relax
    }{%
      \@tempdima#4\relax
    }%
    \parindent\z@ \leftskip#3\relax \advance\leftskip\@tempdima\relax
    \rightskip\@pnumwidth plus4em \parfillskip-\@pnumwidth
    #5\leavevmode\hskip-\@tempdima
      \ifcase #1
       \or\or \hskip 1em \or \hskip 2em \else \hskip 3em \fi%
      #6\nobreak\relax
    \hfill\hbox to\@pnumwidth{\@tocpagenum{#7}}\par
    \nobreak
    \endgroup
  \fi}
\makeatother

\hypersetup{
bookmarks,
bookmarksdepth=3,
bookmarksopen,
bookmarksnumbered,
pdfstartview=FitH,
colorlinks,backref,hyperindex,
linkcolor=Sepia,
anchorcolor=BurntOrange,
citecolor=MidnightBlue,
citecolor=OliveGreen,
filecolor=BlueViolet,
menucolor=Yellow,
urlcolor=OliveGreen
}

\newsavebox{\pullback}
\sbox\pullback{%
\begin{tikzpicture}%
\draw (0,0) -- (1ex,0ex);%
\draw (1ex,0ex) -- (1ex,1ex);%
\end{tikzpicture}}

\newsavebox{\pullbackdl}
\sbox\pullbackdl{%
\begin{tikzpicture}%
\draw (-1ex,0ex) -- (0ex,0ex);%
\draw (0ex,-1ex) -- (0ex,0ex);%
\end{tikzpicture}}

\newsavebox{\pushoutdr}
\sbox\pushoutdr{%
\begin{tikzpicture}%
\draw (-1ex,-1ex) -- (-1ex,0ex);%
\draw (-1ex,0ex) -- (0ex,0ex);%
\end{tikzpicture}}

\newcommand{\cred}{\color{red}}
\newcommand{\cblue}{\color{blue}}

\newcommand{\mustata}{Musta{\c{t}}{\u{a}}}

\newcommand{\cExt}{\mathcal{E}xt}

\newcommand{\bC}{\mathbb{C}}

\newcommand{\bQ}{\mathbb{Q}}

\newcommand{\cH}{\mathcal{H}}

\newcommand{\cM}{\mathcal{M}}
\newcommand{\cO}{\mathcal{O}}

\newcommand{\m}{\mathfrak{m}}
\newcommand{\p}{\mathfrak{p}}

\DeclareMathOperator{\Gr}{Gr}
\DeclareMathOperator{\DR}{DR}

\DeclareMathOperator{\DDB}{\underline{\Omega}}

\DeclareMathOperator{\Spec}{Spec}

\DeclareMathOperator{\Hom}{Hom}

\theoremstyle{plain}
\newtheorem{theorem}{Theorem}[section]

\newtheorem{proposition}[theorem]{Proposition}
\newtheorem{lemma}[theorem]{Lemma}
\newtheorem{corollary}[theorem]{Corollary}

\newtheorem{claim}[theorem]{Claim}
\newtheorem*{claim*}{Claim}

\theoremstyle{definition}
\newtheorem{definition}[theorem]{Definition}

\newtheorem{setting}[theorem]{Setting}

\newtheorem*{setup*}{Setup}

\theoremstyle{remark}
\newtheorem{remark}[theorem]{Remark}

\numberwithin{equation}{theorem}



\newif\ifshowColoursAndTodoes
\showColoursAndTodoestrue

\ifshowColoursAndTodoes
\def\todo#1{\textcolor{Mahogany}%
{\footnotesize\newline{\color{Mahogany}\fbox{\parbox{\textwidth-15pt}{\textbf{todo: } #1}}}\newline}}
\def\commentbox#1{\textcolor{Mahogany}%
{\footnotesize\newline{\color{Mahogany}\fbox{\parbox{\textwidth-15pt}{\textbf{comment: } #1}}}\newline}}

\else
\def\todo#1{}
\def\commentbox#1{}
\renewcommand{\st}[1]{}
\colorlet{red}{black!100} 
\colorlet{teal}{black!100} 
\colorlet{brown}{black!100} 
\colorlet{blue}{black!100} 
\colorlet{magenta}{black!100} 
\colorlet{purple}{black!100} 
\colorlet{cyan}{black!100} 
\fi

\newcommand{\kdot}{{{\,\begin{picture}(1,1)(-1,-2)\circle*{2}\end{picture}\,}}}

\title[Higher singularities for hypersurfaces]{Higher singularities for hypersurfaces}

\author{Mircea \mustata}
\address{Department of Mathematics, University of Michigan, 530 Church Street, Ann Arbor, MI 48109, USA} 
\email{mmustata@umich.edu}

\author{Jakub Witaszek} 
\address{Northwestern University, Department of Mathematics, Lunt Hall, 2033 Sheridan Road, Evanston, IL 60208, USA}
\email{jakub.witaszek@northwestern.edu}

\begin{document}

\begin{abstract}
With an assumption on the codimension of the singular locus of a complex hypersurface $D$ in smooth variety $X$, we show that if $\underline{\Omega}^m_D \cong \Omega^m_D$, then $\underline{\Omega}^i_D \cong \Omega^i_D$ for all $0 \leq i \leq m$. We also discuss an analogue of this statement in positive characteristic. 
\end{abstract}

\maketitle

\setcounter{tocdepth}{2}
\tableofcontents

\section{Introduction}

Recent years have seen significant interest in the Hodge-theoretic properties of complex singularities. A central role in this study is played by the Deligne-Du Bois complexes $\DDB^i_X \in D^b_{\rm coh}(X)$. Recall that when $X$ is smooth, we have $\DDB^i_X=\Omega_X^i[0]$, but when $X$ is singular, these complexes provide the correct replacement for the sheaves of K\"{a}hler differentials, while encoding in a subtle way the singularities. For example, for a singular projective variety $X$ we have a non-canonical isomorphism
\[
H^m(X,\bC) \cong \bigoplus_{i+j=k} H^j(X, \DDB^i_X),
\]
which may serve as an analogue of Hodge decomposition. Similarly, if $L$ is an ample line bundle, then we have 
$$H^q(X,\DDB^p_X\otimes_{\cO_X}L)=0\quad\text{for}\quad p+q>{\rm dim}(X),$$
extending the classical Kodaira-Akizuki-Nakano vanishing theorem.

Using Deligne Du-Bois complexes, one defines \emph{$m$-Du Bois} singularities by the following conditions (see \cite{MOPW23,JKSY22,SVV}):\\[-1em]
\begin{enumerate}
\item ${\rm codim}_X{\rm Sing}(X) \geq 2m+1$, and \\[-0.7em]
\item $\DDB^i_X \cong \Omega^{[i]}_X$ for all $0 \leq i \leq m$.\\[-0.7em]
\end{enumerate}
The study of such singularities has  attracted significant attention, see \cite{MOPW23,JKSY22,Friedman-Laza1,Friedman-Laza2,SVV,popa2024injectivityvanishingdubois,Mustata-Popa22,kawakami2025inversionadjunctionhigherrational,CDM23k-rational-complete-intersection, PP24}, with particular interest devoted to the hypersurface case.  In this case (or more generally, for locally complete intersection singularities),  $X$ is $m$-Du Bois exactly when the canonical morphism $\Omega^i_X \to \DDB^i_X$ is an isomorphism for all $0 \leq i \leq m$. 

Motivated by these developments, it is natural to ask whether the Deligne-Du~Bois complexes exhibit a coherent structural behavior in the hypersurface setting. 
In what follows, assuming a suitable codimension condition on the singular locus of a hypersurface 
$D \subseteq X$ in a smooth variety $X$ over $\bC$, we show that if $\DDB^m_D \cong \Omega^m_D$ for some integer $m$, 
then $\DDB^i_D \cong \Omega^i_D$ for all $0 \leq i \leq m$. 
While this result is of independent interest from the Hodge-theoretic perspective originating in the work of 
Steenbrink and Deligne, the proof relies essentially on recent advances in the theory of higher singularities, 
such as \cite{SVV,popa2024injectivityvanishingdubois,kawakami2025inversionadjunctionhigherrational}. The following is our main result.

\begin{theorem} \label{thm:mainchar0}
Let $D \subseteq X$ be a normal hypersurface in a smooth variety $X$ over $\bC$  and let $m>0$ be an integer. If $m \leq {\rm codim}_D\, {\rm Sing}(D)  -2$ and $\DDB^m_D \cong \Omega^{m}_D$, then
\[
\DDB^i_D \cong \Omega^{i}_D
\]
for every $0 \leq i \leq m$.  
\end{theorem}
\noindent In particular, $D$ is $m$-Du-Bois if and only if ${\rm codim}_D\, {\rm Sing}(D) \geq 2m+1$ and $\DDB^m_D \cong \Omega^{m}_D$. We refer to Theorem \ref{thm:mainStrongest} for the strongest variant of the above theorem.\\

Next, we turn to the arithmetic setting. 
Recently, Kawakami and the second author introduced higher $F$-injective singularities, 
which serve as positive characteristic counterparts of higher Du~Bois singularities.
\begin{definition} \label{def:mFinjective}
Let $D \subseteq X$ be a normal hypersurface of dimension $d$ in a smooth variety $X$ over a perfect field $k$ of characteristic $p>0$, and let $m \geq 0$ be a fixed integer. We say that $D$ is \emph{$m$-$F$-injective} if ${\rm codim}_D\, {\rm Sing}(D) \geq 2m+1$ and
\[
C^{-1} \colon H^{i}_\m(\Omega^{j}_D) \to H^{i}_\m\left(\frac{F_*\Omega^{j}_D}{B\Omega^{j}_D}\right) 
\]
is injective at every closed point $\m \in D$ whenever $j \leq m$ and $i+j \leq d$.
\end{definition}
\noindent We refer to {Section} \ref{ss:Cartier} for the definition of $C^{-1}$, and to the appendix for the discussion on how the above definition relates to the one in \cite{KW24}. Also, let us point out that for hypersurfaces with {small singular locus}, $H^i_\m(\Omega^j_D)=0$ when $i+j < d$ (see Proposition \ref{prop:hypesurfaceVanishing}), and so, in the presence of the first condition in the above definition, the second condition is only interesting for $i+j=d$.


The following result is the analogue of Theorem~\ref{thm:mainchar0} in positive characteristic. 
\begin{theorem} \label{thm:maincharp}
Let $D \subseteq X$ be a normal hypersurface of dimension $d$ in a smooth variety $X$ over a perfect field $k$ {of characteristic $p>0$}, and let $m>0$ be a fixed integer. If $m \leq {\rm codim}_D\, {\rm Sing}(D)-2$ and 
\begin{equation} \label{eq:Cconditionintro}
C^{-1} \colon H^{d-m}_\m(\Omega^{m}_D) \to H^{d-m}_\m\left(\frac{F_*\Omega^{m}_D}{B\Omega^{m}_D}\right) 
\end{equation}
is injective at every closed point $\m \in D$,  then
\[
C^{-1} \colon H^{i}_\m(\Omega^{j}_D) \to H^{i}_\m\left(\frac{F_*\Omega^{j}_D}{B\Omega^{j}_D}\right) 
\]
is injective at every closed point $\m \in D$ whenever $j \leq m$ and $i+j \leq d$.
\end{theorem}
\noindent Thus, $D$ is $m$-$F$-injective if and only if ${\rm codim}_D\,  {\rm Sing}(D) \geq 2m+1$ and (\ref{eq:Cconditionintro})  is injective at all closed points of $D$. 

Let us note that this theorem was the first result we proved. 
Motivated by the arithmetic setting, we subsequently established the characteristic zero theorem, 
which is somewhat more involved (see Remark \ref{remark:Char0vsCharpProof}).

\subsection{Acknowlegements}
The authors thank Brad Dirks, Eamon Quinlan-Gallego,  Tatsuro Kawakami, Theo Sandstrom, and Kevin Tucker for valuable conversations related to the content of the paper.
{\mustata } was supported by NSF grant DMS-2301463 and by the Simons Collaboration grant \emph{Moduli of
Varieties}.
Witaszek was supported by NSF research grants DMS-2101897 and DMS-2401360.

\section{Preliminaries} \label{ss:preliminaries}
We start by summarising some basic definitions.
\begin{enumerate}
\setlength\itemsep{0.3em}
    \item A \emph{variety} over a field $k$ is a {reduced} scheme, which is separated and of finite type over $k$. The dimension of a variety $X$ is the maximum over all the dimensions of $\cO_{X,x}$ for {the} closed points $x \in X$.
    \item A \emph{log resolution} $f \colon Y \to X$ of a normal variety $X$ is a projective birational morphism such that $Y$ is regular and the exceptional divisor $E$ has simple normal crossings (cf.\ \cite[Tag 0BI9]{stacks-project}).
    \item We let $\DDB^i_X \in D^b_{\rm coh}(X)$ denote the $i$-th Du Bois complex of a variety $X$ over $\mathbb{C}$. When $X$ is smooth, we have $\DDB^i_X=\Omega^i_X[0]$. In general, it is given by
\[
\DDB^i_X \coloneqq R\epsilon_{\kdot,*}\Omega^i_{X_\kdot},
\]
where $\epsilon_\kdot \colon X_{\kdot}\to X$ is a hyperresolution (see \cite[Chapter~5]{GNPP} or \cite[Chapter~7.3]{Peter-Steenbrink(Book)} for details). Equivalently, $\DDB^i_X$ is the derived $h$-sheafification of $\Omega^i$ on the $h$-site over $X$. By \cite[Proposition 4.2 (i)]{Huber-Jorder}, $\cH^0(\DDB^i_X)$ is torsion-free. We always have a canonical morphism $\Omega^i_X \to \DDB^i_X$. Whenever we write $\Omega_X^i\simeq\underline{\Omega}_X^i$ in this paper, what we mean is that this canonical map is an isomorphism. 
\end{enumerate}
\vspace{0.5em}
We start with an easy criterion for when a coherent sheaf is reflexive.
\begin{lemma} \label{lem:reflexive}
Let $X = \Spec R$ be a normal affine variety over a field $k$ and let $\cM$ be a coherent sheaf on $X$. Let $U \subseteq X$ be an open subset such that $\cM|_U$ is locally free and $Z := X\smallsetminus U$ is of codimension at least two. If
\[
H^i_\m(\cM) = 0 
\]
for all closed points $\m \in  Z$ and $i \leq \dim Z+1$, then $\cM$ is reflexive.
\end{lemma}
\begin{proof}  
Recall that $\cM$ is reflexive (that is, the canonical map $\cM \to \cM^{**}$ is an isomorphism) if and only if for every, not necessarily closed, point $\p\in X$ the following conditions hold:
\begin{enumerate}
\item[i)] If $\dim(R_{\p})\leq 1$, then ${\cM}_{\p}$ is a free $R_{\p}$-module, and
\item[ii)] If $\dim(R_{\p})\geq 2$, then ${\rm depth}({\cM}_{\p})\geq 2$
\end{enumerate}
(see \cite[Proposition~1.4.1]{BrunsHerzog}). In our setting,
since ${\rm codim}_X\, Z \geq 2$ and ${\cM}|_U$ is locally free, the condition in i) is satisfied, and for the condition in ii), we may assume that $\p\in Z$.

We claim that (cf.\  \cite[Remark 2.5]{kawakami2025inversionadjunctionhigherrational})
\begin{equation} \label{eq:reflexive-claim}
H^i_\p(\cM_{\p}) = 0
\end{equation}
for all $i \leq \dim Z+1 - \dim \overline{\p}$, where ${\overline{\p}}$ denotes the closure of $\p$. To this end, we denote the Matlis duality functor by $(-)^{\vee}$. By local duality and our assumptions, we get
\[
0 = H^i_\m(\cM_\m) = \cExt^{-i}_{\cO_X}(\cM, \omega^\bullet_X)_\m^\vee
\]
for all closed points $\m \in Z$ and $i \leq \dim Z + 1$. In particular,
\[
\cExt^{-i}_{\cO_X}(\cM, \omega^\bullet_X)_\p = 0
\]
for all points $\p \in Z$ and $i \leq \dim Z + 1$. Applying local duality on $R_\p$ then yields
\begin{align*}
H^i_\p(\cM_\p) &\cong \cExt^{- i}_{\cO_{X,\p}}(\cM_\p, \omega^\bullet_{\cO_{X,\p}})^\vee \\
&\cong  \cExt^{- i}_{\cO_X}(\cM, \omega^\bullet_{X}[-\dim \overline{\p}])^\vee_\p \\
&\cong \cExt^{- (i +\dim \overline{\p})}_{\cO_X}(\cM, \omega^\bullet_{X})^\vee_\p \\
&=0
\end{align*}
as long as $i +\dim \overline{\p} \leq \dim Z+1$, which concludes the proof of Claim (\ref{eq:reflexive-claim}).

In particular, the claim gives that 
\[
H^0_\p(\cM_{\p}) = H^1_\p(\cM_{\p}) = 0,
\]
which implies ${\rm depth}({\cM}_{\p})\geq 2$ as required.
\end{proof}
Next, we recall {the} Koszul resolution.
\begin{lemma} \label{LemmaKoszul}
Let $D$ be a hypersurface in a smooth affine variety $X = \Spec R$ defined over a field $k$. Assume that $D = V(f)$, for $f \in R$. For every integer $m$, with $0 < m \leq {\rm codim}_D\, {\rm Sing}(D)$, the natural complex
$$
0 \to \cO_D \xrightarrow{\wedge df} \Omega^1_X|_D \xrightarrow{\wedge df} \ldots \xrightarrow{\wedge df} \Omega^{m}_X|_D \xrightarrow{\rm res} \Omega^m_D \to 0 
$$
is exact. In particular, the complex
$$
0 \to \Omega^{m-1}_D \xrightarrow{\wedge df} \Omega^{m}_X|_D \xrightarrow{\rm res} \Omega^m_D \to 0 
$$
is exact as well.
\end{lemma}

\begin{proof}
The exactness of 
$$0 \to \cO_D \xrightarrow{\wedge df} \Omega^1_X|_D \xrightarrow{\wedge df} \ldots \xrightarrow{\wedge df} \Omega^{m}_X|_D$$
is \cite[Theorem 16.8 and 13.6(i)(iii)]{Matsumura}. The exactness {at} the last two terms is always true and easy to verify.
Finally, the exactness of the short exact sequence is a direct consequence of the first exact sequence.
\end{proof}

\begin{proposition} \label{prop:hypesurfaceVanishing}
Let $D$ be a normal hypersurface of dimension $d$ in a smooth affine variety $X = \Spec R$ defined over a perfect field $k$. Suppose that $D = V(f)$ for $f \in R$. Then for every maximal ideal $\m \in {\rm Sing}(D)$ and $j \leq {\rm codim}_D\, {\rm Sing}(D)$ we have that
\begin{align*}
H^i_\m(\Omega^j_D) &= 0 &&\text{ if $i+j < d$} \\ 
{\rm Soc}(H^i_\m(\Omega^j_D)) &\cong {\rm Soc}(H^d_\m(\cO_D)) \cong {R/\m} &&\text{ if $i+j = d$}.
\end{align*}
In particular, $\Omega^j_D$ is reflexive when $j \leq {\rm codim}_D\, {\rm Sing}(D) - 2$.
\end{proposition}
\noindent We do not calculate this cohomology when $i+j > d$. Also, recall that the socle ${\rm Soc}(M)$ of an $\m^\infty$-torsion module $M$ is defined as the $\m$-torsion submodule $M[\m]$. Finally, we note that the last assertion in the proposition
is due to Vetter, see \cite[Satz 4]{Vetter}.
\begin{proof}
First, we note that the ``in particular'' part follows from  Lemma \ref{lem:reflexive}, with $Z={\rm Sing}(D)$, and the vanishing in the statement of the proposition: 
\[
H^i_\m(\Omega^j_D) =0\quad\text{for}\quad i\leq d-1-({\rm codim}_D\, {\rm Sing}(D) - 2) = \dim {\rm Sing}(D) + 1.
\]
 Since $\Omega^j_X|_D$ is locally free and $D$ is Cohen-Macaulay as it is a hypersurface in a regular variety, we have that
\[
H^i_\m(\Omega^{j}_X|_D) = 0 \quad \text{ for } i<d.
\]

Pick integers $i,j \geq 0$ such that $i+j\leq d$ and $j \leq {\rm codim}_D\, {\rm Sing}(D)$. The case $j=0$ of the proposition is clear, and so we may assume that $j \geq 1$. Consider the short exact sequence from Lemma \ref{LemmaKoszul}:
\[
0 \to \Omega^{j-1}_D \xrightarrow{\wedge df} \Omega^{j}_X|_D \to \Omega^j_D \to 0.
\]
This induces a long exact sequence of local cohomology:
\[
H^i_\m(\Omega^{j}_X|_D) \to H^i_\m(\Omega^j_D) \to H^{i+1}_\m(\Omega^{j-1}_D) \to H^{i+1}_\m(\Omega^{j}_X|_D).\\[0.7em]
\]

\noindent \textbf{Case A: $i+j < d$}. In this situation, since $j \geq 1$, we have that $i < d-1$. In particular,
\[
H^i_\m(\Omega^j_X|_D) = H^{i+1}_\m(\Omega^{j}_X|_D) = 0,
\]
and so by the above long exact sequence:
\[
H^i_\m(\Omega^j_D) = H^{i+1}_\m(\Omega^{j-1}_D).
\]
Hence, by induction on $j$, we get that $H^i_\m(\Omega^j_D) = 0$ if $i +j < d$.\\

\noindent \textbf{Case B: $i+j = d$}. In this situation, the same argument as above shows that
\[
H^i_\m(\Omega^j_D) \cong H^{d-1}_\m(\Omega^1_D),
\]
hence the statement in the proposition reduces to $j=1$. In this case, we have the long exact sequence:
\[
0 \to H^{d-1}_\m(\Omega^1_D) \to H^{d}_\m(\cO_D) \xrightarrow{\theta} H^{d}_\m(\Omega^{1}_X|_D) \to H^d_\m(\Omega^1_D) \to 0,
\]
where $\theta$ is induced by wedging by $df$.

Up to localisation and choosing  algebraic coordinates $x_1,\ldots, x_n$ of $R$ at $\m$, we have a natural trivialisation $\Omega^1_X|_D \cong \cO^{\oplus n}_D$ for which the map $\theta$ is given by 
\[
\theta(\alpha) = \big((\partial_{x_1} f)\alpha, \ldots, (\partial_{x_n} f)\alpha\big).
\]
Since $D$ is Gorenstein, ${\rm Soc}(H^d_\m(\cO_D)) \cong R/\m$. Moreover, as $\m \in {\rm Sing}(D)$, we have that
\[
(\partial_{x_1}f, \ldots, \partial_{x_n}f) \subseteq \m.
\]
Hence $\theta({\rm Soc}(H^d_\m(\cO_D))) = 0$, and so by the above long exact sequence:
\[
{\rm Soc}(H^{d-1}_\m(\Omega^1_D)) \cong {\rm Soc}(H^d_\m(\cO_D)) \cong R/\m
\]
as required. 
\end{proof}
\begin{remark}
It will never be the case however that $H^i_\m(\Omega^j_D)$ is isomorphic to $H^d_\m(\cO_D)$ for $j>0$ and $i+j=d$ on the nose. Indeed, they are Matlis dual to the localisations of ${\rm Ext}_{{\mathcal O}_D}^{d-i}(\Omega^j_D,\omega_D)$ and $\omega_D$ at $\m$, respectively, with the former supported on ${\rm Sing}(D)$ {and} the latter being a line bundle on $D$.
\end{remark}


We shall also use the following result.
\begin{lemma} \label{lem:basic-sesOmega}
If $X$ is a smooth variety over $\bC$ and $E$ is a reduced simple normal crossing divisor
on $X$, then  
\[
\DDB^i_E \cong \Omega^i_E/{\rm tors},
\]
where ${\rm tors}$ denotes the torsion part of $\Omega^i_E$, and {we have a} natural short exact sequence
\[
0 \to \Omega^i_X(\log E)(-E) \to \Omega^i_X \xrightarrow{\rm res} \Omega^i_E/{\rm tors} \to 0.
\]
\end{lemma}
\begin{proof}
See \cite[Corollary 2.16 and Corollary 2.17]{KW24}.
\end{proof}

\subsection{Triangulated categories}

In this subsection, we collect some easy lemmas about triangulated categories which will be needed later on. We refer to \cite{FM06} for an introduction to triangulated categories, and to \cite{HTT} and \cite{BBDG18} for a more detailed treatment of $t$-structures.


\begin{remark}
 Let $\mathcal T$ be a triangulated category. Throughout this article, we will repeatedly use the first axiom of triangulated categories: given a map $\phi \colon A \to B$, there exists an exact triangle
\begin{equation} \label{eq:extension}
A \xrightarrow{\phi} B \xrightarrow{\psi} C \xrightarrow{+1} A[1].
\end{equation}
We call $C$ the \emph{cone} of the map $\phi \colon A \to B$ and denote it by ${\rm Cone}(\phi \colon A \to B)$. Note that $C$ is determined up to an isomorphism by the map $\phi \colon A \to B$, but not always up to a unique isomorphism. Namely, given two cones $C$ and $C'$, the third axiom of triangulated categories shows that there exists a map of exact triangles:
\[
\begin{tikzcd}
A \ar{r}{\phi} \ar{d}{=} & B \ar{r}{\psi} \ar{d}{=} & C \ar[dash]{d}{\alpha} \ar{r}{+1} & A[1]. \ar{d}{=} \\
A \ar{r}{\phi} & B \ar{r}{\psi'} & C' \ar{r}{+1}  & A[1].
\end{tikzcd}
\]
Moreover, such a map $\alpha$ is forced to be a isomorphism. 
However, the choice of $\alpha$ is not unique. We also emphasise that the map $\psi \colon B \to C$ is \emph{not} uniquely determined by $\phi \colon A \to B$. 

Similarly, we define the \emph{cocone} $C$ of a map $\phi \colon A \to B$ denoted ${\rm Cocone}(\phi \colon A \to B)$ as an object fitting inside an exact triangle:
\[
C \xrightarrow{\psi} A \xrightarrow{\phi} B \xrightarrow{+1} C[1].
\]
As above, a cocone is unique up to a (non-unique) isomorphism, and the map $\psi$ is not uniquely determined by the datum of $\phi \colon A \to B$. It is a consequence of the second axiom of triangulated categories that we have an isomorphism
$${\rm Cone}(\phi \colon A \to B)\simeq {\rm Cocone}(\phi\colon A\to B)[1].$$

\end{remark}



\begin{lemma} \label{lem:coneFact1}
Let 
\[
A \xrightarrow{\phi} B \xrightarrow{\psi} C \xrightarrow{+1}
\]
be an exact triangle in a triangulated category $\mathcal T$. Consider a map $\theta \colon D \to B$ in $\mathcal{T}$ such that the composition
\[
D \xrightarrow{\theta} B \xrightarrow{\psi} C
\]
is zero. Then there exists a map $D \to A$ rendering the following diagram commutative:
\[
\begin{tikzcd}
& D \ar[dashed]{dl}[swap]{\exists} \ar{d}{\theta} \ar{dr}{0} & \\
A \arrow{r}{\phi} & B  \arrow{r}{\psi} & C 
\end{tikzcd}
\]
\end{lemma}
\noindent As before, we emphasise that the constructed map is not canonical.
\begin{proof}
This follows by applying ${\rm Hom}_{\mathcal T}(D, -)$ to the exact triangle so that we get an exact sequence
\[
{\rm Hom}_{\mathcal T}(D, A) \to {\rm Hom}_{\mathcal T}(D, B) \to {\rm Hom}_{\mathcal T}(D, C). \qedhere
\]
\end{proof}

\begin{lemma}\label{lem:coneFact2}
Let 
\[
A \xrightarrow{\phi} B \xrightarrow{\psi} C \xrightarrow{+1}
\]
be an exact triangle in a triangulated category $\mathcal T$. Consider a map $\theta \colon B \to D$ in $\mathcal{T}$ such that the composition
\[
A \xrightarrow{\phi} B \xrightarrow{\theta} D
\]
is zero. Then there exists a map $C \to D$ rendering the following diagram commutative:
\[
\begin{tikzcd}
& D    & \\
A \ar{ru}{0} \arrow{r}{\phi} & B \ar{u}{\theta}  \arrow{r}{\psi} & C \ar[dashed]{lu}[swap]{\exists}
\end{tikzcd}
\]
\end{lemma}

\begin{proof}
This follows by applying ${\rm Hom}_{\mathcal T}(-,D)$ to the above exact triangle so that we get an exact sequence
\[
{\rm Hom}_{\mathcal T}(A, D) \leftarrow {\rm Hom}_{\mathcal T}(B,D) \leftarrow {\rm Hom}_{\mathcal T}(C,D). \qedhere
\]
\end{proof}

\begin{lemma} \label{lem:factorisationH0Triangulated}
Let $(\mathcal{T}, \mathcal{T}^{\leq 0}, \mathcal{T}^{\geq 0})$ be a triangulated category with a $t$-structure. If
\[
f \colon A \to B
\]
is a map in $\mathcal{T}$ between $A \in \mathcal T^{\heartsuit}$ and $B \in \mathcal T^{\geq 0}$, then there exists a unique factorisation:
\[
A \xrightarrow{=} {}^t\cH^0(A) \xrightarrow{{}^t\cH^0(f)} {}^t\cH^0(B) \xrightarrow{\rm can} B,
\]
where ${\rm can}$ is the canonical map $\tau^{\leq 0}(B) \to B$. In particular, we have a natural identification:
\[
{\rm Hom}_{\mathcal T}(A,B) \xleftarrow{\cong} {\rm Hom}_{\mathcal T}(A,{}^t\cH^0(B)).
\]
\end{lemma}
\begin{proof}
Consider the following exact triangle:
\[
{}^t\cH^0(B) \xrightarrow{\rm can} B \to {}^t\tau^{>0}(B) \xrightarrow{+1}.
\]
By the axioms of $t$-structures, we have
\[
{\rm Hom}_{\mathcal T}(A,{}^t\tau^{>0}(B))=0 \text{ and } \\
{\rm Hom}_{\mathcal T}(A,{}^t\tau^{>0}(B)[-1])=0.
\]
By applying ${\rm Hom}_{\mathcal T}(A,-)$ to the above exact triangle, we get an exact sequence
\[
0 \! = \! {\rm Hom}_{\mathcal T}(A,\! {}^t\tau^{>0}(B)[-1] ) \!\to\! {\rm Hom}_{\mathcal T}(A,\!{}^t\cH^0(B) ) \!\to\! {\rm Hom}_{\mathcal T}(A,\! B) \!\to\! {\rm Hom}_{\mathcal T}(A,\! {}^t\tau^{>0}(B)) \! = \! 0
\]
of abelian groups. This concludes the proof. 
\end{proof}

\begin{lemma} \label{lem:megafunctorialityadjunction}
Let $u \colon Z \to X$ be a projective morphism of  Noetherian schemes of finite dimension with $X$ regular. Let $A, B \in D^b_{\rm coh}(X)$ and $C, D \in D^b_{\rm coh}(Z)$. Consider the following commutative diagram:
\begin{equation} \label{eq:mega1}
\begin{tikzcd}[row sep = 2em, column sep = 4em]
 Ru_*D  & \ar{l}{\beta} B    \\
 Ru_*C  \ar{u}{Ru_*\theta} & \ar{l}{\alpha} A.   \ar{u}{\gamma} 
\end{tikzcd}
\end{equation}
Denote the images of $\alpha$ and $\beta$ under the adjunctions:
\begin{align*}
\Hom_{D^b_{\rm coh}(X)}(A, Ru_*C) &\cong \Hom_{D^b_{\rm coh}(Z)}(Lu^*A, C) \\
\Hom_{D^b_{\rm coh}(X)}(B, Ru_*D) &\cong \Hom_{D^b_{\rm coh}(Z)}(Lu^*B, D)
\end{align*}
by $\widetilde{\alpha}$ and $\widetilde \beta$, respectively. Then the following diagram is commutative:
\begin{equation} \label{eq:mega2}
\begin{tikzcd}[row sep = 2em, column sep = 4em]
 D  & \ar{l}{\widetilde \beta} Lu^*B    \\
 C  \ar{u}{\theta} & \ar{l}[swap]{\widetilde \alpha} Lu^*A.   \ar{u}{Lu^*\gamma} 
\end{tikzcd}
\end{equation}
\end{lemma}
\begin{proof}
Under the assumptions of the lemma, the functor $Lu^*$ is left adjoint to $Ru_*$ (see \cite[II.5.11]{HartshorneResidues}, cf.\ \cite[p.\ 83, Compatiblities (3)]{FM06}). Then Diagram (\ref{eq:mega2}) is constructed from Diagram (\ref{eq:mega1}) by applying $Lu^*$ and composing with the counit maps: $Lu^*Ru_*(-) \xrightarrow{\eta} (-)$ (see also \cite[Tag 0DVC]{stacks-project}).
\end{proof}


\subsection{The Cartier operator} \label{ss:Cartier}
We review the theory of Cartier operators as discussed in \cite[Section 2.3]{KW24}.
In what follows, $X$ is a variety over a perfect field $k$ of characteristic $p>0$. In particular, $F_*\Omega_X^\bullet \in D^b_{\rm coh}(X)$ and we can define coherent sheaves:
\begin{align*}
B\Omega^i_X &:= {\rm im}(d \colon F_*\Omega^{i-1}_X \to F_*\Omega^i_X), \text{ and }\\
Z\Omega^i_X &:= \ker(d \colon F_*\Omega^{i}_X \to F_*\Omega^{i+1}_X).
\end{align*}
They fit in the following short exact sequence:
\begin{equation} \label{eq:defBZ}
0 \to Z\Omega^i_X \to F_*\Omega^i_X \xrightarrow{d} B\Omega^{i+1}_X \to 0.
\end{equation}
Next, we construct the \emph{inverse Cartier operator}
\begin{equation} \label{eq:inverse-Cartier-basic}
C^{-1} \colon \Omega^i_X \to \frac{Z\Omega^i_X}{B\Omega^i_X},
\end{equation}
following \cite[Section 1.3]{fbook}. The definition being local, we may assume that $X = \Spec R$ for a finitely generated $k$-algebra $R$ and consider $\gamma \colon R \to \Omega^1_R$ given by the formula $\gamma(r) = r^{p-1}dr$ for $r \in R$. In view of the following identities (see \cite[Lemma 1.3.3]{fbook}):
\begin{enumerate}
\item $\gamma(rs) = r^p\gamma(s) + \gamma(r)s^p$,
\item $d\gamma(r)=0$, and
\item $\gamma(r+s) - \gamma(r) - \gamma(s) \in B\Omega^1_R$,
\end{enumerate}
we obtain the induced map on differential forms:
\begin{align*}
\gamma \colon \Omega^1_R \to \frac{Z\Omega^1_R}{B\Omega^1_R} \qquad 
rds \mapsto r^ps^{p-1}ds.
\end{align*}
Finally, (\ref{eq:inverse-Cartier-basic}) can be obtained by taking exterior powers.
Set
 \[
 G\Omega^i_X := \frac{F_*\Omega^i_X}{B\Omega^i_X}.
 \]
 In this paper, we shall often abuse the notation and denote the composition
\[
\Omega^i_X \to \frac{Z\Omega^i_X}{B\Omega^i_X} \hookrightarrow \frac{F_*\Omega^i_X}{B\Omega^i_X} = G\Omega^i_X 
\]
by $C^{-1}$ as well.

The key theorem of Cartier (\cite[Theorem 1.3.4]{fbook}) stipulates that when $X$ is 
{smooth}, the inverse Cartier operator (\ref{eq:inverse-Cartier-basic}) is an isomorphism for every $i$. Therefore, we get short exact sequences
\begin{equation} \label{eq:CartierOperator}
0 \to B\Omega^i_X \to Z\Omega^i_X \xrightarrow{C} \Omega^i_X \to 0,
\end{equation}
where $C \colon Z\Omega^i_X \to \Omega^i_X$ is called the \emph{Cartier operator}.  It is a well-known fact that $Z\Omega^i_X$ and $B\Omega^i_X$ are locally free in this case (cf.\ \cite[Lemma 5.10]{KTTWYY1}). Moreover, by (\ref{eq:defBZ}), we get the following short exact sequence:
\[
0 \to \Omega^i_X \cong \frac{Z\Omega^i_X}{B\Omega^i_X} \to G\Omega^i_X \xrightarrow{d} B\Omega^{i+1}_X \to 0.
\]
In particular, $G\Omega^i_X$ is locally free as well.

\section{Proof of the main theorem in characteristic $0$} \label{ss:char0}

First, we define higher Du-Bois singularities and recall  their characterisation  from \cite[Proposition 3.4]{kawakami2025inversionadjunctionhigherrational}. 

\begin{definition}[{\cite{SVV}}]
Let $X$ be a normal connected variety over $\bC$ and let $m \geq 0$ be an integer. We say that $X$ is \emph{pre-$m$-Du-Bois} if 
\[
\DDB^i_X \cong \cH^0(\DDB^{i}_{X})
\]
for all integers $0 \leq i \leq m$. 

We say that $X$ is \emph{$m$-Du-Bois} if $\DDB^i_X \cong \Omega^{[i]}_{X}$ for all integers $0 \leq i \leq m$ and
 \[
 {\rm codim}_X({\rm Sing}(X)) \geq 2m+1.
 \]
\end{definition}

\begin{theorem}[{\cite[Proposition 3.4]{kawakami2025inversionadjunctionhigherrational}}] \label{thm:SplitKW}
If $X$ is a normal, connected, $d$-dimensional affine variety over $\bC$, then $X$ has  pre-$m$-Du-Bois singularities if and only if the natural maps
\[
 H^j_\m(\cH^0(\DDB^{i}_X))  \to H^j_\m(\DDB^i_X) 
\]
are injective for every maximal ideal $\m$, $0 \leq i \leq m$, and $i+j \leq d$.
\end{theorem}

Next we recall the following well-known consequence of the decomposition theorem.

\begin{lemma} \label{lem:splitDT}
If $X$ is a smooth variety over $\bC$ and $f \colon Y \to X$ is a projective birational morphism such that $Y$ is smooth, too, then the natural maps
\[
f^* \colon \Omega^i_X \to Rf_* \Omega^i_Y
\]
split for every integer $i \geq 0$.
\end{lemma}
\begin{proof}
By Saito's decomposition theorem for Hodge modules (\cite[(4.5.2) and (4.5.4)]{Saito90}), we have
\[
Rf_*\bQ^H_Y[d] = \bQ^H_X[d] \oplus N,
\]
for some $N \in D^b({\rm MHM}(X))$. Therefore, the map
\begin{align*}
\Omega^i_X[d-i] &= \Gr_{-i}\DR(\bQ^H_X[d]) \\
&\longrightarrow \Gr_{-i}\DR(Rf_*\bQ^H_Y[d]) \\
&= Rf_*\Gr_{-i}\DR(\bQ^H_Y[d]) \\
&= Rf_*\Omega^i_Y[d-i]
\end{align*}
splits as required. We refer to \cite[Proposition 2.17 and Remark 2.18]{kawakami2025inversionadjunctionhigherrational} for the properties of the ${\rm Gr}_\bullet {\rm DR}$ functor used in the above argument.
\end{proof}

\begin{setting} \label{eq:char0setting}
Let $X$ be a smooth variety over $\bC$, let $D$ be a reduced divisor, and let $f \colon Y \to X$ be a log resolution of $(X,D)$ which is an isomorphism over $X\smallsetminus D$. Set $E := f^*(D)_{\rm red}$. For every $i \geq 0$, we define natural restriction maps $\phi$ and $\psi$:
\begin{align}
&\phi \colon \Omega^i_X \xrightarrow{{\rm res}} \Omega^i_D \xrightarrow{{\rm can}}  \DDB^i_D, \text{ and }  \label{eq:phipsi} \nonumber \\
&\psi \colon \Omega^i_Y \xrightarrow{{\rm res}} \Omega^i_E \xrightarrow{{\rm can}}  \DDB^i_E \cong \Omega^i_E/{\rm tor}, \nonumber
\end{align}
where  ${\rm can}$ denotes the natural map from a sheaf to its derived $h$-sheafification, and the isomorphism $\DDB^i_E \cong \Omega^i_E/{\rm tor}$ is discussed in Lemma \ref{lem:basic-sesOmega}. 

Finally, we fix splittings $\alpha_i \colon Rf_*\Omega^i_Y \to \Omega^i_X$
of the pull-back maps $f^* \colon \Omega^i_X \to Rf_*\Omega^i_Y$ coming from Lemma \ref{lem:splitDT}. We shall often drop the subscript $i$, and just denote this map by $\alpha$.
\end{setting}

\begin{lemma} \label{lem:DDBProperty}
With notation as in Setting \ref{eq:char0setting}, there exists a map $\tau \colon Rf_*\Omega^i_Y(\log E)(-E) \to  \Omega^i_X$ which makes the following diagram commutative, with the rows being exact triangles:
\[
\begin{tikzcd}
\hphantom{a} & \ar{l}{+1} Rf_*\Omega^i_E/{\rm tor}  & \ar{l}{Rf_*\psi} Rf_*\Omega^i_Y  & \ar{l}{Rf_*\tau_Y} Rf_*\Omega^i_Y(\log E)(-E) \\
\hphantom{a} & \ar{l}{+1} \DDB^i_D  \ar{u}{f^*} & \ar{l}{\phi} \Omega^i_X  \ar{u}{f^*} &  \ar{l}{\tau} Rf_*\Omega^i_Y(\log E)(-E).  \ar{u}{=}  
\end{tikzcd}
\]
Here $\tau_Y$ is the natural inclusion $\Omega^i_Y(\log E)(-E) \to \Omega^i_Y$.

Moreover, there exists a map $\beta \colon Rf_*\Omega^i_E/{\rm tor} \to \DDB^i_D$ which renders the following diagram commutative as well:
\[
\begin{tikzcd}
\hphantom{a} & \ar{l}{+1} Rf_*\Omega^i_E/{\rm tor} \ar[dashed]{d}{\beta} & \ar{l}{Rf_*\psi} Rf_*\Omega^i_Y  \ar{d}{\alpha} & \ar{l}{Rf_*\tau_Y} Rf_*\Omega^i_Y(\log E)(-E) \ar{d}{=} \\
\hphantom{a} & \ar{l}{+1} \DDB^i_D   & \ar{l}{\phi} \Omega^i_X  &  \ar{l}{\tau} Rf_*\Omega^i_Y(\log E)(-E).
\end{tikzcd}
\]
\end{lemma}

We warn the reader that while the choice of $\alpha$ is dependent on the decomposition theorem only, the formation of $\beta$ is not canonical or uniquely determined by $\alpha$ in any capacity. We only get from the proof the commutativity $\beta \circ Rf_*\psi = \phi \circ \alpha$. On the other hand, this diagram shows that $\tau$ is determined by $\alpha$, namely $\tau = \alpha \circ Rf_*\tau_Y$.

\begin{proof}
We have the following diagram: 
\[
\begin{tikzcd}
E \ar[hook]{r} \ar{d}{f|_E} & Y \ar{d}{f} \\
    D \ar[hook]{r} & X.
\end{tikzcd}
\]
Since $\DDB^i_X$ is a derived $h$-sheaf, this induces the following homotopy pullback square:
\[
\begin{tikzcd}
Rf_*\DDB^i_E   & \ar{l}{Rf_*\psi} Rf_*\Omega^i_Y  \\
\DDB^i_D  \ar{u}{f^*} & \ar{l}{\phi} \Omega^i_X.  \ar{u}{f^*} 
\end{tikzcd}
\]
In particular, by the standard property of homotopy pullback squares, the cofiber of the lower horizontal map is naturally isomorphic to the cofiber of the top horizontal map. Since $\DDB^i_E  \cong \Omega^i_E/{\rm tor}$, where ${\rm tor}$ denotes the torsion part of $\Omega^i_E$, and
\[
{\rm ker}(\Omega^i_Y \to \Omega^i_E/{\rm tor}) \cong \Omega^i_Y(\log E)(-E)
\]
(cf.\ Lemma \ref{lem:basic-sesOmega}), we get the following diagram in which rows are exact triangles:
\[
\begin{tikzcd}[column sep = large]
\hphantom{a} & \ar{l}{+1} Rf_*\DDB^i_E   & \ar{l}{Rf_*\psi} Rf_*\Omega^i_Y  
& \ar{l}{ Rf_*\tau_Y}  Rf_*\Omega^i_Y(\log E)(-E)  \\
\hphantom{a} & \ar{l}{+1} \DDB^i_D  \ar{u}{f^*} & \ar{l}{\phi} \Omega^i_X  \ar{u}{f^*} &  \ar{l}{\tau} Rf_*\Omega^i_Y(\log E)(-E).  \ar{u}{=}  
\end{tikzcd}
\]
Recall that $\tau_Y$ is the natural inclusion  $\Omega^i_Y(\log E)(-E) \to \Omega^i_Y$. The reader might also construct the above diagram without discussing homotopy pullback squares, but instead using the octahedral axiom.

Now, Lemma \ref{lem:splitDT} provides us with the aforementioned  splitting
\[
\alpha \colon Rf_*\Omega^i_Y \to \Omega^i_X
\]
of $f^* \colon \Omega^i_X \to Rf_*\Omega^i_Y$. Next, we check that the following diagram commutes:
\[
\begin{tikzcd}[column sep = large]
Rf_*\Omega^i_Y  \ar{d}{\alpha} & \ar{l}{Rf_*\tau_Y} Rf_*\Omega^i_Y(\log E)(-E) \ar{d}{=} \\
 \Omega^i_X   &  \ar{l}{\tau} Rf_*\Omega^i_Y(\log E)(-E).  
\end{tikzcd}
\]
This is clear by the following composition
\[
\begin{tikzcd}[column sep = large]
Rf_*\Omega^i_Y(\log E)(-E) \ar{rd}{\tau}   \ar{r}{=} & Rf_*\Omega^i_Y(\log E)(-E) \ar{r}{Rf_*\tau_Y} & Rf_*\Omega^i_Y \ar{r}{\alpha} & \Omega^i_X.  \\
& \Omega^i_X \ar{ur}{f^*} \ar[bend right=10]{urr}{{\rm id}} & & 
\end{tikzcd}
\]
Hence $\tau = \alpha \circ Rf_*\tau_Y$.

Therefore, by the third axiom of triangulated categories we get the sought-after map $\beta \colon Rf_*\DDB^i_E \to \DDB^i_D$ fitting inside the following commutative diagram:
\[
\begin{tikzcd}[column sep = large]
\hphantom{a} & \ar{l}{+1} Rf_*\DDB^i_E \ar[dashed]{d}{\beta} & \ar{l}{Rf_*\psi} Rf_*\Omega^i_Y  \ar{d}{\alpha} & \ar{l}{Rf_*\tau_Y} Rf_*\Omega^i_Y(\log E)(-E) \ar{d}{=} \\
\hphantom{a} & \ar{l}{+1} \DDB^i_D  & \ar{l}{\phi} \Omega^i_X   &  \ar{l}{\tau} Rf_*\Omega^i_Y(\log E)(-E).
\end{tikzcd}
\]
In particular, $\phi\circ\alpha=\beta\circ Rf_*\psi$ as required. This concludes the proof of the lemma.

\end{proof}

 In particular, we get the following.
\begin{corollary} \label{cor:afterrestriction}
With the notation in Setting \ref{eq:char0setting} and Lemma \ref{lem:DDBProperty},
if $f^*D$ denotes the scheme-theoretic pullback of $D$, then the following two diagrams are commutative:
\[
\begin{tikzcd}[row sep = 2em, column sep = 4em]
 Rf_*\DDB^i_E & \ar{l}{Rf_*\psi_D} Rf_*(\Omega^i_Y|_{f^*D})    \\
 \DDB^i_D  \ar{u}{f^*} & \ar{l}{\phi_D} \Omega^i_X|_D. \ar{u}{f^*} 
\end{tikzcd}
\quad \text{ and } \quad
\begin{tikzcd}[row sep = 2em, column sep = 4em]
 Rf_*\DDB^i_E \ar{d}{\beta} & \ar{l}{Rf_*\psi_D} Rf_*(\Omega^i_Y|_{f^*D}) \ar{d}{\alpha_D}   \\
 \DDB^i_D   & \ar{l}{\phi_D} \Omega^i_X|_D. 
\end{tikzcd}
\]
where $\phi_D$ and $\psi_D$ are restrictions of $\phi$ and $\psi$ to $D$ and $f^*D$, respectively, and $\alpha_D := \alpha \otimes^L_{\cO_X} \cO_D$ is a splitting.
\end{corollary}
\begin{proof}
This follows from the commutativity of the left squares in Lemma \ref{lem:DDBProperty} in view of Lemma \ref{lem:megafunctorialityadjunction} with $u \colon D \hookrightarrow X$ being the natural closed immersion:
\[
\begin{tikzcd}[row sep = 2em, column sep = 4em]
 Ru_*Rf_*\DDB^i_E \ar[dashed, bend left = 30]{d}{Ru_*\beta} & \ar{l}{Rf_*{\psi}} Rf_*\Omega^i_Y \ar[dashed, bend left = 30]{d}{\,{\alpha}}   \\
 Ru_*\DDB^i_D  \ar{u}{Ru_*f^*} & \ar{l}{{\phi}} \Omega^i_X   \ar{u}{f^*} 
\end{tikzcd}
\qquad \overset{\text{adjunction}}{\rightsquigarrow} \qquad
\begin{tikzcd}[row sep = 2em, column sep = 4em]
 Rf_*\DDB^i_E \ar[dashed, bend left = 30]{d}{\beta} & \ar{l}{Rf_* \psi_D} Lu^* Rf_*\Omega^i_Y \ar[dashed, bend left = 30]{d}{\, Lu^*\alpha}   \\
 \DDB^i_D  \ar{u}{f^*} & \ar{l}{ \phi_D} Lu^*\Omega^i_X  \ar{u}{Lu^*f^*}     
\end{tikzcd}
\]
Here $Lu^*\Omega^i_X = \Omega^i_X \otimes_{\cO_X} \cO_D = \Omega^i_X|_D$ by the local freeness of $\Omega^i_X$ and 
\[
Lu^*Rf_*\Omega^i_Y = Rf_*\Omega^i_Y \otimes_{\cO_X}^L \cO_D \cong Rf_*(\Omega^i_Y \otimes_{\cO_Y} \cO_{f^*D}) = Rf_*(\Omega^i_Y|_{f^*D})
\]
by the projection formula \cite[Tag 0B54]{stacks-project} and the local freeness of $\Omega^i_Y$.
\end{proof}

\begin{lemma} \label{lem:MirceaCalculation}
Let $(R,\m)$ be {the local ring} of a smooth affine variety over $\bC$ at a closed point. Let $x_1,\ldots, x_n \in R$ be a lift of an $R/\m$-basis of $\m/\m^2$. 
If $g = ux_1^{a_1}\cdots x_m^{a_m}$, for $u \in R^\times$ and integers $a_1, \ldots, a_m \geq 1$, and we put $D = V(g)$ and $E= V(x_1\cdots x_m)$, then  
 the composition of natural maps:
\[
\Omega^{k-1}_R(\log E)(-E) \hookrightarrow \Omega^{k-1}_R \xrightarrow{\wedge dg} \Omega^k_R \xrightarrow{{\rm res}} \Omega^k_R|_D
\]
is zero for every $k >  0$.
\end{lemma}

\begin{proof}
Up to replacing $R$ by its completion, we may assume that $R = \bC \llbracket x_1,\ldots, x_n \rrbracket$.
Pick a form $\omega$ in $\Omega^{k-1}_R(\log E)(-E)$:
\[
\omega = fx_1\cdots x_m\left(\prod_{j \in J}\frac{dx_j}{x_j}\right)\left(\prod_{i \in I}dx_i\right)
\]
with $J \subseteq [1,m]$, $I \subseteq [m+1,n]$, and $f \in R$ such that $|J| + |I| = k-1$. The image of this form in $\Omega^k_R|_{D}$, namely:
\[
fx_1\cdots x_m\left(\prod_{j \in J}\frac{dx_j}{x_j}\right)\left(\prod_{i \in I}dx_i\right) \wedge dg 
\]
is equal to:
\begin{multline*}
fx_1^{a_1+1}\cdots x_m^{a_m+1}du\left(\prod_{j \in J}\frac{dx_j}{x_j}\right)\left(\prod_{i \in I}dx_i\right) + \\ \sum_{t=1}^m ua_tfx_1^{a_1+1}\cdots x_m^{a_m+1}\frac{dx_t}{x_t}\left(\prod_{j \in J}\frac{dx_j}{x_j}\right)\left(\prod_{i \in I}dx_i\right)
\end{multline*}
which, by direct inspection, is divisible by $x_1^{a_1}\cdots x_m^{a_m}$. Hence, this image is zero in $\Omega^k_R|_{D}$ concluding the proof. \qedhere
\end{proof}

The next proposition is the key technical component of the proof of the main theorem in characteristic zero. Before we state it, we need to introduce additional notation building on that of Setting \ref{eq:char0setting} and Corollary \ref{cor:afterrestriction}. Set 
\begin{equation} \label{eq:K}
K := {\rm Cocone}(\phi_D \colon \Omega^i_X|_D \to \DDB^i_D),
\end{equation}
namely, we fix an exact triangle:
\begin{equation} \label{eq:KExactTriangle}
\begin{tikzcd}
K \ar{r}{u} & \Omega^i_X|_D \ar{r}{\phi_D} & \DDB^i_D \ar{r}{+1} & . \hphantom{a}
\end{tikzcd}
\end{equation}
Assume that $i \leq {\rm codim}_D\, {\rm Sing}(D)$. By the third axiom of triangulated categories, there exists a map of exact triangles:
\begin{equation} \label{eq:keydiagramChar0}
\begin{tikzcd}
K \ar{r}{u} & \Omega^i_X|_D \ar{r}{\phi_D} & \DDB^i_D \ar{r}{+1} & \hphantom{a} \\
\Omega^{i-1}_D \ar[dashed]{u}{\gamma} \ar{r}{\wedge dg} & \Omega^i_X|_D \ar{u}[swap]{=} \ar{r}{\rm res} & \Omega^i_D \ar{u}{{\rm can}} \ar{r}{+1} & .\hphantom{a}
\end{tikzcd}
\end{equation}
Here, the lower exact triangle is given by Lemma \ref{LemmaKoszul}. 

\begin{remark} \label{rem:gammaUnique}
Assume that $D$ is normal and $i \leq {\rm codim}_D\, {\rm Sing}(D) -2$.  Given that $\cH^0(\DDB^{i}_D)$ is torsion-free, Proposition \ref{prop:hypesurfaceVanishing} shows that the natural maps 
\[
\Omega^{i}_D \to \cH^0(\DDB^{i}_D) \to \Omega^{[i]}_D
\]
are all isomorphisms. The same applies to $\Omega^{i-1}_D \to \cH^0(\DDB^{i-1}_D) \to \Omega^{[i-1]}_D$. We claim that the map $\gamma$ is uniquely determined in this case. 

To this end, note that $K\in D^{\geq 0}_{\rm coh}(D)$. Moreover, $\cH^0(u) \colon \cH^0(K) \to \Omega^i_X|_D$ is an injection in the abelian category of coherent sheaves on $D$. Thus from the diagram:
\[
\begin{tikzcd}[column sep = large]
0 \ar{r} & \cH^0(K) \ar{r}{\cH^0(u)} & \Omega^i_X|_D \ar{r}{\cH^0(\phi_D)} & \cH^0(\DDB^i_D) \\
0 \ar{r} & \Omega^{i-1}_D \ar{u}{\cH^0(\gamma)} \ar{r}{\wedge dg} & \Omega^i_X|_D \ar{u}[swap]{=} \ar{r}{\rm res} & \Omega^i_D \ar{u}{{\rm can}}[swap]{\cong} 
\end{tikzcd}
\]
we can infer that the map $H^0(\gamma) \colon \Omega^{i-1}_D \to \cH^0(K)$ is uniquely determined\footnote{To put it simply, once (\ref{eq:KExactTriangle}) is fixed, we can canonically identify 
\[
0 \to \cH^0(K) \xrightarrow{\cH^0(u)} \Omega^i_X|_D \xrightarrow{\cH^0(\phi_D)} \cH^0(\DDB^i_D) \quad \text{ with } \quad 0 \to \Omega^{i-1}_D \xrightarrow{\wedge dg} \Omega^i_X|_D \xrightarrow{\rm res} \Omega^i_D.
\]
}. 
By Lemma \ref{lem:factorisationH0Triangulated}, there exists a factorisation:
\[
\gamma \colon \Omega^{i-1}_D \xrightarrow{\cong} \cH^0(K) \xrightarrow{\rm can} K,
\]
where ${\rm can}$ identifies with the natural map $\tau^{\leq 0}(K) \to K$. This concludes the claim that $\gamma$ is uniquely determined.

\end{remark}

\begin{proposition} \label{prop:keydiagramchar0}
With the above notation (see Setting \ref{eq:char0setting} and Corollary \ref{cor:afterrestriction}) we assume that $D=V(g)$ is a normal divisor of dimension $d$ in a smooth variety $X = \Spec R$ for some $g \in R$. If $i>0$ is a fixed integer such that $i \leq {\rm codim}_D\, {\rm Sing}(D)-2$, then
there exists a map $\underline{\gamma} \colon \DDB^{i-1}_D \to K$ fitting inside a factorisation (see (\ref{eq:keydiagramChar0})):
\[
\gamma \colon \Omega^{i-1}_D \xrightarrow{\rm can} \DDB^{i-1}_D \xrightarrow{\underline{\gamma}} K.
\] 
\end{proposition}

\begin{proof}
Recall that we denote the scheme-theoretic inverse image of $D$ by $f^*D$  and its reduction by $E$. Also, recall that we have a surjective restriction map
\[
\psi_D \colon \Omega^i_Y|_{f^*D} \xrightarrow{{\rm res}} \Omega^i_E \xrightarrow{{\rm can}}  \DDB^i_E = \Omega^i_E/{\rm tor}
\]
(cf.\ Lemma \ref{lem:basic-sesOmega}), which fits inside a short exact sequence:
\[
0 \to {\rm ker}(\psi_D) \xrightarrow{v} \Omega^i_Y|_{f^*D} \xrightarrow{\psi_D} \Omega^i_E/{\rm tor} \to 0.
\]
In the proof of the proposition we will use the following two commutative diagrams:
\begin{equation} \label{eq:BigDiag1}
\begin{tikzcd}[column sep = large]
\hphantom{a} & \arrow{l}{+1} Rf_*\DDB^i_{E} & \arrow{l}{Rf_*\psi_D} Rf_*(\Omega^i_Y|_{f^*D})  & \arrow{l}{Rf_*v} Rf_*\ker(\psi_D)  \\
\hphantom{a} & \arrow{l}{+1} \DDB^i_D \arrow{u}{f^*} & \arrow{l}{\phi_D} \Omega^i_X|_D \arrow{u}{f^*} & \arrow{l}{u} K. \arrow[dashed]{u} 
\end{tikzcd}
\end{equation}
\begin{equation} \label{eq:BigDiag2}
\begin{tikzcd}[column sep = large]
\hphantom{a} & \arrow{l}{+1} Rf_*\DDB^i_{E} \ar{d}{\beta} & \arrow{l}{Rf_*\psi_D} Rf_*(\Omega^i_Y|_{f^*D}) \ar{d}{\alpha_D} & \arrow{l}{Rf_*v} Rf_*\ker(\psi_D) \ar[dashed]{d}  \\
\hphantom{a} & \arrow{l}{+1} \DDB^i_D  & \arrow{l}{\phi_D} \Omega^i_X|_D  & \arrow{l}{u} K.  
\end{tikzcd}
\end{equation}
The left squares come from Corollary \ref{cor:afterrestriction}. The existence of the maps between $K$ and $Rf_*{\rm ker}(\psi)$ follows from the third axiom of triangulated categories.


Next, we construct maps $s \colon \Omega^{i-1}_X \to K$ and $t \colon \Omega^{i-1}_Y \to {\rm ker}(\psi_D)$, which render the following diagram commutative:
\begin{equation} \label{eq:almostBiggestDiagram}
\begin{tikzcd}[column sep = large]
  Rf_*\DDB^i_{E}  & \arrow{l}{Rf_*\psi_D} Rf_*(\Omega^i_Y|_{f^*D})  & \arrow{l}{Rf_*v} Rf_*\ker(\psi_D)  & \arrow[dashed]{l}{Rf_*t} Rf_*\Omega_Y^{i-1} \ar[bend right=20]{ll}[swap]{\wedge d(g \circ f)}  \\
 \DDB^i_D \arrow{u}{f^*} & \arrow{l}{\phi_D}  \Omega^i_X|_D \arrow{u}{f^*} & \arrow{l}{u} K \arrow{u} &  \arrow{u}{f^*} \arrow[dashed]{l}{s} \Omega_X^{i-1}. \ar[bend left = 20]{ll}[swap]{\wedge dg} 
\end{tikzcd}
\end{equation}
By Lemma \ref{lem:coneFact1}, to construct the dashed arrows $s$ and $t$ in this diagram, it is enough to show that the compositions
\begin{align*}
&\Omega^{i-1}_X \xrightarrow{\wedge dg} \Omega^i_X|_D \xrightarrow{{\rm res}} \Omega^i_D \xrightarrow{{\rm can}} \DDB^i_D, \text{ and } \\
&\Omega^{i-1}_Y \xrightarrow{\wedge d(g \circ f)} \Omega^i_Y|_{f^*D} \xrightarrow{{\rm res}} \Omega^i_E \xrightarrow{{\rm can}} \DDB^i_E
\end{align*}
are zero, which is clear. Then, what is left, is showing that the rightmost square commutes. But in view of Lemma \ref{lem:factorisationH0Triangulated}, since $\Omega^{i-1}_X$ is a sheaf and $Rf_*\ker(\psi_D) \in D^{\geq 0}$, it is enough to verify the commutativity after applying $\cH^0$, in which case this is clear,
as the formations of $f^*$ and $\wedge dg$ commute with one another, and $\cH^0(K)$ and $f_*\ker(\psi_D)$ inject into $\Omega^i_X|_D$ and $f_*(\Omega^i_Y|_{f^*D})$, respectively.

Now, putting everything together, we can consider the following commutative diagram in the derived category of coherent sheaves $D^b_{\rm coh}(X)$:
\[
\begin{tikzcd}[column sep = 3em, row sep = large]
 Rf_*(\Omega^i_Y|_{f^*D})  & \arrow{l}{Rf_*v} Rf_*\ker(\psi_D)  & \arrow{l}{Rf_*t} Rf_*\Omega_Y^{i-1}   \ar[bend right=30]{ll}[swap]{\wedge d(g \circ f)}  & \ar{l}{ Rf_*\tau_Y} \ar[dashed, bend right = 30]{ll}[swap]{0} Rf_*\Omega^{i-1}_Y(\log E)(-E) \\
 \Omega^i_X|_D \arrow{u}{f^*} & \arrow{l}{u} K \arrow{u} &  \arrow{u}{f^*} \arrow{l}{s} \Omega_X^{i-1} \ar[bend left = 30]{ll}[swap]{\wedge dg} & \ar{l}{\tau} Rf_*\Omega^{i-1}_Y(\log E)(-E) \ar{u}{=}.
\end{tikzcd}
\]
The squares are commutative by Lemma \ref{lem:DDBProperty} and Diagram (\ref{eq:almostBiggestDiagram}). We shall verify that the dashed arrow in the above diagram is zero. To this end, we observe that by Lemma \ref{lem:MirceaCalculation}, the composition:
\[
\begin{tikzcd}
\Omega^{i-1}_Y(\log E)(-E)  \ar[hook]{r}{\tau_Y} & \Omega^{i-1}_Y \ar{r}{t} \ar[bend right = 25]{rr}{\wedge dg} & {\rm ker}(\psi_D) \ar[hook]{r}{v} & \Omega^i_Y|_{f^*D}
\end{tikzcd}
\]
is zero. Hence, the composite map $t \circ \tau_Y \colon \Omega^{i-1}_Y(\log E)(-E) \to \ker(\psi_D)$ is zero, and the same holds true after applying $Rf_*$, thereby justifying that the dashed arrow 
\[
Rf_*(t \circ \tau_Y) \colon Rf_*\Omega^{i-1}_Y(\log E)(-E) \to Rf_*{\rm ker}(\psi_D)
\]
in the above diagram is indeed zero. 
\begin{claim}
The composition
\[
K \xleftarrow{\ s \ } \Omega^{i-1}_X \xleftarrow{\ \tau \ } Rf_*\Omega^{i-1}_Y(\log E)(-E)
\]
is zero.
\end{claim}
\begin{proof}
Recall that $\cH^0(K) \cong \Omega^{i-1}_D$ (cf.\ Remark \ref{rem:gammaUnique}). To prove the claim it is enough to justify that the following diagram, where the leftmost map comes from (\ref{eq:BigDiag2}), is commutative:
\[
\begin{tikzcd}[column sep = large]
 Rf_*\ker(\psi_D) \ar{d} & \arrow{l}{Rf_*t} Rf_*\Omega_Y^{i-1}  & \ar{l}{Rf_*\tau_Y} \ar[dashed, bend right = 20]{ll}[swap]{0} Rf_*\Omega^{i-1}_Y(\log E)(-E) \\
 K  &  \arrow{u}{f^*} \arrow{l}{s} \Omega_X^{i-1}  & \ar{l}{\tau} Rf_*\Omega^{i-1}_Y(\log E)(-E) \ar{u}{=}.
\end{tikzcd}
\]
Here, the rightmost square is commutative by Lemma \ref{lem:DDBProperty} and the leftmost square is commutative by Lemma \ref{lem:factorisationH0Triangulated}, given that it is so after the application of $\cH^0$. Namely, in this case the diagram specialises to:
\[
\begin{tikzcd}[column sep = large]
f_*(\Omega^i_Y|_{f^*D}) \ar[dashed]{d}{\text{splitting}} & f_*\ker(\psi_D) \ar[dashed, hook']{l}{f_*v} \ar{d} & \ar{l}{f_*t} \ar[dashed, bend right = 20]{ll}[swap]{\wedge 
{ d(g \circ f)}} f_*\Omega^{i-1}_Y  \\
\Omega^i_X|_D & \ar[dashed, hook']{l}{u} \underbrace{\cH^0(K)}_{\cong\, \Omega^{i-1}_D} & \ar{l}{s}   \ar[dashed, bend left = 35]{ll}{\wedge dg} \ar{u}{f^*} \Omega^{i-1}_X.
\end{tikzcd}
\]
which is clearly commutative as $f^*\eta \wedge {d(g \circ f)} = f^*(\eta \wedge dg)$ for any local form $\eta$ in $\Omega^{i-1}_X$. Here, the dashed arrows in the diagram are added to justify the commutativity.

\end{proof}

Recall that $\DDB^{i-1}_D \cong {\rm Cone}(\tau)$ by  Lemma \ref{lem:DDBProperty}.  Thus, by the above claim and  Lemma \ref{lem:coneFact2}, we get the existence of a map $\underline{\gamma} \colon \DDB^{i-1}_D \to K$
fitting inside the following commutative diagram:
\begin{equation} \label{eq:almostLastDiagram}
\begin{tikzcd}[column sep = large]
 K  & \ar{l}{s} \Omega^{i-1}_X \ar{dl}{\phi} &  \ar{l}{\tau} Rf_*\Omega^{i-1}_Y(\log E)(-E)   \\
 \DDB^{i-1}_D. \ar[dashed]{u}{\underline \gamma} & & 
\end{tikzcd}
\end{equation}
Recall that $\phi$ factorises into 
\[
\Omega^{i-1}_X \xrightarrow{\rm res} \Omega^{i-1}_D \xrightarrow{\rm can} \DDB^{i-1}_D,
\]
and we thus get the following commutative diagram:
\[
\begin{tikzcd}[column sep = large]
 K  & \ar{l}{s} \Omega^{i-1}_X  \ar{d}{\rm res}   \\
 \DDB^{i-1}_D \ar{u}{\underline \gamma} &  \ar{l}{\rm can}\Omega^{i-1}_D \ar[dashed]{ul}{\gamma}.
\end{tikzcd}
\]
The commutativity of (\ref{eq:almostLastDiagram}) and the definition of $s$ from (\ref{eq:almostBiggestDiagram}) insure that $\cH^0(\underline{\gamma}) = \cH^0(\gamma)$, and so the dashed arrow in the above diagram is indeed equal to $\gamma$ thanks to Lemma \ref{lem:factorisationH0Triangulated} (cf.\ Remark \ref{rem:gammaUnique}). Therefore, we have a factorisation
\[
\gamma \colon \Omega^{i-1}_D \xrightarrow{\rm can} \DDB^{i-1}_D \xrightarrow{\underline{\gamma}} K
\]
as required.  \qedhere
\end{proof}

We are ready to give the proof of our main theorem in characteristic zero. In fact, we shall prove a more general statement. 

\begin{theorem}[cf. Theorem \ref{thm:mainchar0}] \label{thm:mainStrongest}
 Let $D$ be a normal hypersurface of dimension $d$ in a smooth variety $X = \Spec R$ over $\bC$ and let $m$ be a positive integer. If {$m \leq {\rm codim}_D\, {\rm Sing}(D) -2$} and the map
 \[
 {\rm can} \colon H^{d-m}_\m(\Omega^{m}_D)  \to H^{d-m}_\m(\DDB^m_D) 
 \] 
 is injective for every closed point $\m \in D$, then $D$ is $m$-Du-Bois.
\end{theorem}
\begin{proof}
We may assume that $D = V(g)$ for $g \in R$. Throughout the proof, we will implicitly use the fact  $\Omega^i_D$ is reflexive for $i\leq m$ (cf.\ Proposition \ref{prop:hypesurfaceVanishing}). 
In particular, given that $\cH^0(\DDB^i_D)$ is torsion free (see (3) at the beginning of Section~\ref{ss:preliminaries}) and the restriction of the canonical morphism
$\Omega_D^i\to {\mathcal H}^0(\underline{\Omega}_D^i)$ to the smooth locus of $D$ is an isomorphism, it follows that $\Omega_D^i\simeq {\mathcal H}^0(\underline{\Omega}_D^i)$
for all $i \leq m$. In particular, to prove the theorem it is enough to argue that $D$ is pre-$m$-Du-Bois.

Let $f \colon Y \to X$ be a log resolution of $(X,D)$ that is an isomorphism over $X\smallsetminus D$, and let $E = f^{-1}(D)_{\rm red}$. We will argue by descending induction on $i \leq m+1$ that  for every closed point $\m\in D$, the natural map
\begin{equation} \label{eq:goal}
 {\rm can} \colon  H^{d-i+1}_\m(\Omega^{i-1}_D)  \to H^{d-i+1}_\m(\DDB^{i-1}_D) 
\end{equation}
is injective. This immediately concludes the proof of the theorem given the characterisation of  pre-$m$-Du-Bois singularities from Theorem \ref{thm:SplitKW} and the vanishing $H^j_\m(\Omega^{i}_D)=0$ for $i \leq m$ and $i+j < d$  from Proposition \ref{prop:hypesurfaceVanishing}.

The base case $i=m+1$ of our goal $(\ref{eq:goal})$ follows by our assumption. Thus, in order to prove that the above map (\ref{eq:goal}) is injective, we may assume that $1 \leq i \leq m$ and the natural map
\begin{equation} \label{eq:indassump}
 {\rm can} \colon  H^{d-i}_\m(\Omega^{i}_D)  \to H^{d-i}_\m(\DDB^{i}_D) 
\end{equation}
is injective. 


With the notation as in Proposition \ref{prop:keydiagramchar0}, set 
\[
K = {\rm Cocone}(\phi_D \colon \Omega^{i}_X|_D \to \DDB^{i}_D).
\]
Fix $j := d-i$ and consider the following long exact sequence of local cohomology coming from (\ref{eq:keydiagramChar0}):
\[
\begin{tikzcd}
\mathllap{0 = }\ H^{j}_\m(\Omega^{i}_X|_D) \ar{r} &  H^{j}_\m(\DDB^{i}_D) \ar{r} &  H^{j+1}_\m(K) \ar{r} & H^{j+1}_\m(\Omega^{i}_X|_D)  \\
\mathllap{0 = }\ H^{j}_\m(\Omega^{i}_X|_D) \ar{u}{=} \ar{r} &  H^{j}_\m(\Omega^{i}_D) \ar{u}{\rm can}[swap]{(\star)} \ar{r} &  H^{j+1}_\m(\Omega^{i-1}_D)  \ar{u}{\gamma}[swap]{(\star\star)} \ar{r} & H^{j+1}_\m(\Omega^{i}_X|_D).  \ar{u}[swap]{=}
\end{tikzcd}
\]
Here $H^{j}_\m(\Omega^{i}_X|_D)=0$, because $\Omega^{i}_X$ is locally free, $j< d$, and $D$, as a hypersurface in a regular variety, is Cohen-Macaulay. Now, since the map $(\star)$ is injective by the inductive assumption (\ref{eq:indassump}), we get that $(\star\star)$ is injective as well by tracing through the above diagram. 

The injectivity of $(\star\star)$ and the factorisation from Proposition \ref{prop:keydiagramchar0} then implies that the natural map (\ref{eq:goal}):
\[
 {\rm can} \colon H^{j+1}_\m(\Omega^{i-1}_D)  \to H^{j+1}_\m(\DDB^{i-1}_D) 
\]
is injective. This concludes the proof by induction.
\end{proof}

\begin{proof}[{Proof of Theorem \ref{thm:mainchar0}}]
Since $m \leq {\rm codim}_D\, {\rm Sing}(D) -2$ and $\DDB^m_D \cong \Omega^m_D$, the assumptions of Theorem \ref{thm:mainStrongest} are satisfied, and so $D$ is $m$-Du-Bois.  \qedhere

\end{proof}

\section{Proof of the main theorem in characteristic $p>0$} \label{ss:char0}

We start this section with the following lemma.

\begin{lemma} \label{lem:charpDiagram}
Let $D \subseteq X$ be a hypersurface in a smooth affine variety $X = \Spec R$ defined over a perfect field of characteristic $p>0$. If $D = V(f)$, for some $f \in R$, then we have a commutative diagram
\begin{equation} \label{eq:charpDiagramNew}
\begin{tikzcd}
 G\Omega^{i-1}_D \ar{r}{\sigma} & G\Omega^i_X|_D \ar{r}{\rm res} & G\Omega^i_D  \\
\Omega^{i-1}_D \ar{r}{\wedge df} \ar{u}{C^{-1}} & \Omega^i_X|_D 
\ar{r}{\rm res} \ar{u}{C^{-1}} & \Omega^i_D, \ar{u}{C^{-1}}
\end{tikzcd}
\end{equation}
in which the top row is a complex: ${\rm res} \circ \sigma = 0$.
The definitions of the maps in the top row will be explained in the proof. 
\end{lemma}
\noindent Note that under additional assumptions on the singular locus of $D$, the lower row extends to a short exact sequence (see Lemma \ref{LemmaKoszul}), but the upper row will almost never be exact. 
\begin{proof}
Let us define the maps in the top row of the diagram. \\

\noindent \textbf{The restriction map}. First, the map 
\[
{\rm res} \colon G\Omega^i_X|_D \to G\Omega^i_D
\]
is induced by the natural restriction map $F_*\Omega^i_X \to F_*\Omega^i_D$,
which sends $B\Omega^i_X$ to $B\Omega^i_D$ by the definition of boundaries.\\

\noindent \textbf{The map $\sigma$}. Next, the map 
\begin{equation} \label{eq:sigmamap}
\sigma \colon G\Omega^{i-1}_D \to G\Omega^{i}_X|_D
\end{equation}
is induced by the map $F_*\Omega^{i-1}_X \to F_*\Omega^{i}_X$ given by wedging by $F_*f^{p-1}df$. To construct $\sigma$ formally, we first consider the following diagram with the upper row exact
\[
\begin{tikzcd}
  (F_*\Omega^{i-1}_X)|_D \ar{rd}{0} \ar{r}{\cdot F_*f} & (F_*\Omega^{i-1}_X)|_D \ar{r} \ar{d}{\wedge F_*f^{p-1}df} &  F_*(\Omega^{i-1}_X|_D) \ar[dashed, bend left = 30]{dl}{\varsigma} \ar{r} & 0 \\ 
  & (F_*\Omega^i_X)|_D & &
\end{tikzcd}
\]
where the diagonal map is zero, because $F_*f \wedge F_*f^{p-1}df = fF_*df$ which maps to zero in $(F_*\Omega^i_X)|_D$. From this diagram, we infer the existence of the dashed map
\[
\varsigma \colon F_*(\Omega^{i-1}_X|_D) \to (F_*\Omega^i_X)|_D.
\]
In turn, this map fits in the following diagram:
\[
\begin{tikzcd}
  F_*\Omega^{i-2}_D \ar{rd}{0} \ar{r}{\wedge F_*df} & F_*(\Omega^{i-1}_X|_D) \ar{r} \ar{d}{\varsigma} &  F_*\Omega^{i-1}_D \ar[dashed, bend left = 30]{dl}{\sigma} \ar{r} & 0. \\ 
  & (F_*\Omega^i_X)|_D & &
\end{tikzcd}
\]
with the top row being exact by definition of sheaves of differential forms. Here the diagonal map is zero because $df \wedge df = 0$. This yields:
\[
\sigma \colon F_*\Omega^{i-1}_D \to (F_*\Omega^i_X)|_D.
\]
Finally, we claim that this map sends $B\Omega^{i-1}_D$ into $(B\Omega^i_X)|_D$; granted that we get the sought-after map  (\ref{eq:sigmamap}). To prove the claim, choose a local differential form $\eta$ in $\Omega^{i-2}_X$. Then
\[
d\eta \wedge f^{p-1}df = d(\eta \wedge f^{p-1}df),
\]
which immediately implies the claim as every form in $B\Omega^{i-1}_D$ lifts to a form in $B\Omega^{i-1}_X$. 

By construction, the top row in Diagram \ref{eq:charpDiagramNew} is a complex: we have ${\rm res} \circ \sigma = 0$.
\\

\noindent \textbf{Commutativity of squares}. The right square in Diagram (\ref{eq:charpDiagramNew}) is clearly commutative by definition of the inverse Cartier operator. The commutativity of the left square follows since $C^{-1}$ respects wedge product, which is one of its defining properties:
\[
C^{-1}(\eta \wedge df) = C^{-1}(\eta) \wedge C^{-1}(df) = C^{-1}(\eta) \wedge f^{p-1}df. \qedhere
\]
\end{proof}

\begin{proof}[Proof of Theorem \ref{thm:maincharp}]
 Let us fix a closed point $\m\in D$.
As mentioned already, Proposition~\ref{prop:hypesurfaceVanishing} gives
\[
H^i_\m(\Omega^j_D) = 0
\]
for $j \leq m$ and $i+j<d$, hence we only need to consider the case when $i+j=d$. 

We argue by descending induction on $j$. Namely, we assume that
\begin{equation} \label{eq:posindassump}
C^{-1} \colon H^{d-j}_\m(\Omega^{j}_D) \to H^{d-j}_\m\left(G\Omega^j_D\right) 
\end{equation}
is injective for some $j$, with $1 \leq j \leq m$, and aim for showing that
\begin{equation} \label{eq:posindgoal}
C^{-1} \colon H^{d-j+1}_\m(\Omega^{j-1}_D) \to H^{d-j+1}_\m\left(G\Omega^{j-1}_D\right) 
\end{equation}
is injective.

We freely use the notation from Lemma \ref{lem:charpDiagram}. Let
\[
K := {\rm ker}\left(G\Omega^{j}_X|_D \xrightarrow{\rm res} G\Omega^j_D \right). 
\]
We get an induced diagram
\[
\begin{tikzcd}
0 \ar{r} & K \ar{r}{u} & G\Omega^j_X|_D \ar{r}{\rm res} & G\Omega^j_D \ar{r} & 0 \\
0 \ar{r} & \Omega^{j-1}_D \ar{r}{\wedge df} \ar{u}{\theta} & \Omega^j_X|_D 
\ar{r}{\rm res} \ar{u}{C^{-1}} & \Omega^j_D \ar{r} \ar{u}{C^{-1}}  & 0.
\end{tikzcd}
\]
with both rows being exact (see Lemma \ref{LemmaKoszul}).

By Lemma \ref{lem:charpDiagram}, we have a natural factorisation:
\[
\theta \colon \Omega^{j-1}_D \xrightarrow{C^{-1}} G\Omega^{j-1}_D \to K.
\]
Thus, to prove (\ref{eq:posindgoal}), it is enough to show that:
\begin{equation} \label{eq:posindgoal2}
\theta \colon H^{d-j+1}_\m(\Omega^{j-1}_D) \to H^{d-j+1}_\m\left(K\right) 
\end{equation}
is injective. In fact, since every  non-zero element of an $\m^\infty$-torsion module can be multiplied so that it is a non-zero element in the socle (the $\m$-torsion submodule), it is enough to check that $\theta$ induces an injection of socles.

To this end, we consider the induced long exact sequences of local cohomology:
\[
\begin{tikzcd}
0 = H^{d-j}_\m(G\Omega^j_X|_D)  \ar{r}{\rm res} & H^{d-j}_\m(G\Omega^j_D) \ar{r} & H^{d-j+1}_\m(K) \ar{r}{u}  & H^{d-j+1}_\m(G\Omega^j_X|_D) \\
0 = H^{d-j}_\m(\Omega^j_X|_D) \ar{u}{C^{-1}} \ar{r}{\rm res} & H^{d-j}_\m(\Omega^j_D) \ar[hook]{u}{C^{-1}} \ar{r} & H^{d-j+1}_\m(\Omega^{j-1}_D) \ar{r}{\wedge dg} \ar{u}{\theta} & H^{d-j+1}_\m(\Omega^j_X|_D) \ar{u}{C^{-1}}
\end{tikzcd}
\]

The vanishings
\[
H^{d-j}_\m(G\Omega^j_X|_D) =  H^{d-j}_\m(\Omega^j_X|_D)  = 0
\]
follow from the fact that $j\geq 1$ and the sheaves  $G\Omega^j_X$ and $\Omega^j_X$ are locally free, given that $X$ is smooth (see the last paragraph of Section \ref{ss:Cartier}), while $D$ is Cohen-Macaulay. Moreover, the vertical map 
\[
C^{-1} \colon H^{d-j}_\m(\Omega^{j}_D) \to H^{d-j}_\m\left(G\Omega^j_D\right) 
\]
is an injection by the inductive assumption  (\ref{eq:posindassump}).

Now the sought-after injectivity of the $\theta$ map (\ref{eq:posindgoal2}) on socles follows via a diagram chase, using the fact that by Proposition \ref{prop:hypesurfaceVanishing}, the map
\[
{\rm Soc}(H^{d-j}_\m(\Omega^j_D)) \hookrightarrow {\rm Soc}(H^{d-j+1}_\m(\Omega^{j-1}_D))
\]
is an injection of one-dimensional $R/\m$-vector spaces, and so it is an isomorphism.
\end{proof}

\begin{remark} \label{remark:Char0vsCharpProof}
Note that we ran more or less the same argument in characteristic $0$ for the proof of Theorem \ref{thm:mainchar0}, contingent upon the construction of the following diagram:
\begin{equation} \label{eq:char0KeyDiagram}
\begin{tikzcd}[row sep = 0.6em]
& {\rm Cocone}(\DDB^i_X\!|_D \to \DDB^i_D) \ar{r} & \DDB^i_X\!|_D \ar{r} & \DDB^i_D \ar{r}{+1} & \hphantom{a} \\
\DDB^{i-1}_D \ar[dashed]{ru} & & & & \\
& \Omega^{i-1}_D \ar{ul} \ar{uu} \ar{r}{\wedge dg} & \Omega^i_X|_D \ar{uu} \ar{r}{\rm res} & \Omega^i_D \ar{uu} \ar{r}{+1} & \hphantom{a},
\end{tikzcd}
\end{equation}
which is the analogue of the characteristic $p>0$ diagram (cf.\ Lemma \ref{lem:charpDiagram}):
\[
\begin{tikzcd}[row sep = 0.6em]
& {\rm ker}(G\Omega^i_X|_D \to G\Omega^i_D) \ar{r} & G\Omega^i_X|_D \ar{r} & G\Omega^i_D \ar{r}{+1} & \hphantom{a} \\
G\Omega^{i-1}_D \ar[dashed]{ru} & & & & \\
& \Omega^{i-1}_D \ar{ul} \ar{uu} \ar{r}{\wedge dg} & \Omega^i_X|_D \ar{uu} \ar{r}{\rm res} & \Omega^i_D \ar{uu} \ar{r}{+1} & \hphantom{a}
\end{tikzcd}
\]
However, constructing (\ref{eq:char0KeyDiagram}) is more difficult due to the fact that checking factorisations in triangulated categories is significantly more demanding than in abelian categories. The hardest part of the characteristic $0$ proof was the careful construction of the dashed arrow in (\ref{eq:char0KeyDiagram}). 
\end{remark}

\section{Appendix} \label{s:appendix}

In this appendix, we study  the reflexivity and  the  local cohomology of the
 sheaves $B\Omega^i_X$ and $G\Omega^i_X$ featured in the study of higher $F$-injectivity. This is then used to compare Definition \ref{def:mFinjective} with the one in \cite{KW24}.

Throughout this section, we work under the following assumptions:

\begin{setting} \label{setting:appendix}
Let $D \subseteq X$ be a hypersurface of dimension $d$ in a smooth affine variety $X = \Spec R$ defined over a perfect field of characteristic $p>0$. Assume that $D = V(g)$ for $g \in R$. We set an integer $k := {\rm codim}_{D}\, {\rm Sing}(D) - 2$.
\end{setting}

By Proposition \ref{prop:hypesurfaceVanishing}, we have that $\Omega^j_D$ is reflexive for all $0 \leq j \leq k$. In particular, by \cite[Proposition 2.4]{KW24} the reflexivisation for the Cartier operator from the smooth locus yields a short exact sequence:
\begin{equation} \label{eq:sesCartierRefl}
0 \to B\Omega^{[j]}_D \to Z\Omega^{[j]}_D \to \Omega^{j}_D \to 0
\end{equation}
for every $0 \leq j \leq k$.  Here, the first two terms are reflexivisations of $B\Omega^j_D$ and $Z\Omega^j_D$, respectively. In fact, \cite[Proposition 2.4]{KW24} also stipulates that $Z\Omega^j_D = Z\Omega^{[j]}_D$ for all $j \leq k-1$, but we will recover an even stronger statement in the course of the proof below.
\begin{proposition} \label{prop:BZVanshingReflexive}
With assumptions as in Setting \ref{setting:appendix}, we have that $B\Omega^j_D$ and $Z\Omega^j_D$ are reflexive for all $j \leq k$. Moreover:
\begin{equation} \label{eq:goalZBvan}
H^i_\m(B\Omega^{j}_D) = H^i_\m(Z\Omega^{j}_D) = 0
\end{equation}
for all closed points $\m \in D$ whenever $j \leq k$ and $i+j < d$.
\end{proposition}
\begin{proof}

We argue by induction on $j \leq k$. The case $j=0$ is clear as $B\Omega^0_D = 0$ by definition and $Z\Omega^0_D = \cO_D$ by (\ref{eq:sesCartierRefl}).  Thus, from now on, we may assume that $B\Omega^{j-1}_D$ and $Z\Omega^{j-1}_D$ are reflexive and
\begin{equation} \label{eq:assumZBvan}
H^i_\m(B\Omega^{j-1}_D) = H^i_\m(Z\Omega^{j-1}_D) = 0
\end{equation}
whenever $i+j-1 < d$. Our goal is to establish (\ref{eq:goalZBvan}) and show that both $B\Omega^j_D$ and $Z\Omega^j_D$ are reflexive. Fix $0 < j \leq k$ and $i \geq 0$ such that $i+j<d$.

Consider the short exact sequence which defines boundaries and cycles:
\[
0 \to Z\Omega^{j-1}_D \to F_*\Omega^{j-1}_D \xrightarrow{d} B\Omega^{j}_D \to 0.
\]
From this short exact sequence, assumption (\ref{eq:assumZBvan}), and Proposition \ref{prop:hypesurfaceVanishing} we deduce:
\begin{equation} \label{eq:Bvan}
H^i_\m(B\Omega^{j}_D) = 0
\end{equation}
when $i+j < d$. In particular, this vanishing holds for $i \leq d-1-k = {\rm dim}\, {\rm Sing}(D)+1$, and so $B\Omega^j_D$ is reflexive by Lemma \ref{lem:reflexive}.

Next, consider the short exact sequence (\ref{eq:sesCartierRefl}):
\[
0 \to B\Omega^{j}_D \to Z\Omega^{[j]}_D \to \Omega^{j}_D \to 0.
\]
Again, by Proposition \ref{prop:hypesurfaceVanishing} and the vanishing (\ref{eq:Bvan}), we get 
\[
H^i_\m(Z\Omega^{[j]}_D) = 0
\]
when $i+j<d$. In particular, this vanishing holds for $i \leq d-1-k = {\rm dim}\, {\rm Sing}(D)+1$, and so $Z\Omega^j_D$ is reflexive by Lemma \ref{lem:reflexive}. This concludes the proof of the inductive step: the reflexivity of $B\Omega^j_D$ and $Z\Omega^j_D$ as well as the vanishing (\ref{eq:goalZBvan}).
\end{proof}

One can push this argument one step further to show that $H^i_\m(B\Omega^j_D) = 0$ for $i+j<d$ and $j \leq k+1 = {\rm codim}\, {\rm Sing}(D)-1$.\\



In particular, Proposition \ref{prop:BZVanshingReflexive} allows one to compare Definition \ref{def:mFinjective} with the one from \cite[Definition 1.1]{KW24}, which covers the case of isolated singularities.

 \bibliographystyle{amsalpha}
 \bibliography{bibliography}

\end{document}